\documentclass[11pt,a4paper]{amsart}
\usepackage[top=4cm,bottom=4cm,left=3.5cm,right=3.5cm]{geometry}
\usepackage{subcaption}
\usepackage{tabularx}
\usepackage{amsmath,amssymb,amsthm}
\usepackage[english]{babel}
\usepackage{mathtools}
\usepackage{enumerate}
\usepackage[numbers]{natbib}
\usepackage[hidelinks]{hyperref}
\usepackage{tikz}
\usepackage{tkz-graph}
\usepackage{tikz-cd}
\tikzset{%
    symbol/.style={%
        ,draw=none
        ,every to/.append style={%
            edge node={node [sloped, allow upside down, auto=false]{$#1$}}}
    }
}
\usepackage{xcolor}

\usepackage{algorithm}
\usepackage[noend]{algpseudocode}

\DeclareMathOperator{\Ker}{Ker}
\DeclareMathOperator{\Mat}{Mat}

\DeclareMathOperator{\Gal}{Gal}

\newtheorem{theorem}{Theorem}

\newtheorem{lemma}[theorem]{Lemma}
\newtheorem{proposition}[theorem]{Proposition}
\newtheorem{corollary}[theorem]{Corollary}
\theoremstyle{definition}

\newtheorem{definition}[theorem]{Definition}
\theoremstyle{remark}
\newtheorem{remark}[theorem]{Remark}

\numberwithin{equation}{section}

\begin{document}

\title[On rectangular unimodular matrices over the algebraic integers]{On Rectangular unimodular matrices over the ring of algebraic integers}

\author[G. Micheli]{Giacomo Micheli}
\address{Mathematical Institute\\
University of Oxford \\
Woodstock Rd\\ 
Oxford OX2 6GG, United Kingdom
}
\email{giacomo.micheli@maths.ox.ac.uk}

\author[V. Weger]{Violetta Weger}
\address{Institute of Mathematics\\
University of Zurich\\
Winterthurerstrasse 190\\
8057 Zurich, Switzerland\\
}
\email{violetta.weger@math.uzh.ch}

\thanks{The first author is thankful to Swiss National Science Foundation grant number 171248.}
%%%    General info
\subjclass[2010]{}


\maketitle

\begin{abstract}
Let $n$ and $m$ be positive integers such that $n<m$. In this paper we compute the density of rectangular  unimodular $n$ by $m$  matrices over the ring of algebraic integers of a number field.
\end{abstract}

\section{Introduction}
\label{sec:introduction}

Given an $n \times m$ matrix $M$ over an integral domain $\mathcal{R}$, with $n < m$, one of the natural questions arising is whether one can adjoin $m-n$ rows to this matrix, such that the new $m \times m$ matrix $\overline{M}$ is invertible, i.e. the determinant of $\overline{M}$ is a unit in $\mathcal{R}$. Matrices with this property are called rectangular unimodular matrices. It is worth mentioning that the rectangular unimodularity property has been deeply studied in the past: one of the most distinguished results is  the Quillen-Suslin Theorem \cite{projmod,suslin1974projective},
previously known under the name of Serre's Conjecture. This paper deals with the problem of computing the ``probability'' that a randomly chosen rectangular matrix is unimodular. 
Since no uniform probability distribution is allowed over $\mathbb Z$, one introduces the notion of density.
The density of a set $S \subset \mathbb{Z}^d$ is defined to be 
\begin{equation*}
\rho(S) = \lim_{B \rightarrow \infty} \frac{|S \cap [-B,B[^d|}{(2B)^d}
\end{equation*}
if the limit exists.  
In \cite{bib:maze2011natural} the authors compute the density of rectangular unimodular matrices over the ring of rational integers $\mathbb Z$. Unfortunately, the proof of \cite[Proposition 1]{bib:maze2011natural} is flawed as the inequality \cite[Equation 2]{bib:maze2011natural} is wrong: a counterexample is simply obtained by setting $S_i=\{i\}$ for all $i\in \mathbb Z$ (other counterexamples are of course possible). We should notice that \cite[Equation 2]{bib:maze2011natural} is correct for example for finite unions (which is not the case in the proof of \cite[Proposition 1]{bib:maze2011natural}).

The purpose of this paper is twofold: on one side we fix the proof of \cite{bib:maze2011natural} and on the other side we generalise the result to unimodular matrices over the ring of algebraic integers.

For the case $n=1,m=2$ and $\mathcal R=\mathbb Z$ results in this direction were found independently by Ces\'{a}ro \cite{ce3, ce1} and by Mertens \cite{mertens1874ueber}, who proved that the density of  coprime pairs of integers is equal to $\frac{1}{\zeta(2)}$, where $\zeta$ denotes the  Riemann zeta function. 
This has been generalised in \cite{bib:nymann1972probability} to coprime $m$-tuples of integers and in \cite{ferraguti2016mertens} to coprime $m$-tuples of algebraic integers of a number field $K$. Interestingly similar to the result of Mertens-Ces\'{a}ro, the density of coprime $m$-tuples of algebraic integers is $\frac{1}{\zeta_K(m)}$, where $\zeta_K$ denotes the Dedekind zeta function of $K$. Recently, the result in \cite{ferraguti2016mertens} was used  in \cite{albrecht2016subfield}  for a cryptanalysis of the NTRU cryptosystem.

More recent results on densities can be found in  \cite{dotti2016eisenstein, ferraguti2018set, guo2013probability, lieb2018uniform, maze2011natural, micheli2017local,  micheli2016densityeis,   micheli162,  micheli2016density, bib:BS, wang2017smith}.

The main purpose of this paper is to show that the density of $n \times m$ rectangular unimodular matrices over the ring of algebraic integers  $\mathcal{O}_K$ of a number field $K$ is
\begin{equation*}
\prod_{i=0}^{n-1} \frac{1}{\zeta_K(m-i)}.
\end{equation*}

The main tool for proving this is the local to global principle by  Poonen and Stoll \cite[Lemma 1]{bib:loctoglob}, where it is shown that the computation of some densities can be reduced to measuring certain appropriately chosen subsets of the $p$-adic integers. This technique is an evolution of the Ekedahl's Sieve \cite{torsten1991infinite}. 
An even more general result is \cite[Proposition 3.2]{bright2016failures}.

The paper is organised as follows: in Section \ref{sec2} we recall some basic facts from algebraic number theory and introduce the fundamental tool used in the proof of the main result. 
 In Section \ref{sec3} we  state and prove Theorem \ref{density_unimodular_matrices}, the main theorem of this paper.
 Finally, in Section \ref{sec4} we present  some experiments which show how accurate the density result is compared with the actual probability of selecting a random unimodular matrix with bounded entries.

\section{Preliminaries}\label{sec2}

We first want to fix some notations and state some important facts from algebraic number theory.
Let $\mathbb{N}$ be the set of positive integers.
Let $K$ be a number field with degree $[K: \mathbb{Q}]= k$. 
Let us denote by $\mathcal{O}_K$ the ring of algebraic integers of $K$. 
Note that $\mathcal{O}_K \cong \mathbb{Z}^k$ as $\mathbb{Z}$-modules, hence $\mathcal{O}_K$ has a $\mathbb{Z}$-basis, called integral basis.
Every non-zero prime ideal $\mathfrak{p}$ of $\mathcal{O}_K$ intersects $\mathbb{Z}$ in a prime ideal $p\mathbb{Z}$ for some prime $p$. In this case we say that $\mathfrak{p}$ is lying above $p$ and denote this by $\mathfrak{p} \mid p$. The residue field $\mathcal{O}_K/\mathfrak{p}$ is a finite extension of $\mathbb{Z}/p\mathbb{Z} \cong \mathbb{F}_p$: in fact,  
$\mathcal{O}_K/\mathfrak{p} \cong \mathbb{F}_{p^{\deg(\mathfrak{p})}}$, where $\deg(\mathfrak{p}) = [\mathcal{O}_K/\mathfrak{p} : \mathbb{F}_p]$ denotes the inertia degree, for $\mathfrak{p} \mid p$.

\begin{definition}
Let $K$ be a number field and let $L$ be the Galois closure of $K/\mathbb{Q}$. The norm map for the extension field $K / \mathbb{Q}$ is defined as
\begin{eqnarray*}
N_{K/\mathbb{Q}}: K &\to & \mathbb{Q} \\
 a & \mapsto & \prod_{\sigma \in \Gal(L / \mathbb{Q})} \sigma(a).
\end{eqnarray*}
Let $ \ell$ be a positive integer.
Observe that the norm can be extended to the field of rational functions  $K(x_1, \ldots, x_{\ell})$, since it holds that 
\begin{equation*}
\Gal(L / \mathbb{Q}) \cong \Gal( L(x_1, \ldots, x_{\ell}) / \mathbb{Q}(x_1, \ldots, x_{\ell})).
\end{equation*}
For the sake of completeness, let us make explicit the definition of such norm for the extension field $K(x_1, \ldots, x_{\ell})/\mathbb{Q}(x_1, \ldots, x_{\ell})$. For this we need the multi-index notation: let $a = (a_1, \ldots, a_{\ell}) \in \mathbb{N}^{\ell}$, then we denote by 
\begin{eqnarray*}
| a | &=& \sum_{i=1}^{\ell} a_i, \\
x^a &=& \prod_{i=1}^{\ell}x_i^{a_i}.
\end{eqnarray*} 
Let $\varphi \in K[x_1, \ldots, x_{\ell}]$, let $c_a=c_{(a_1, \ldots, a_{\ell})} \in K$, then one can write $\varphi$ as 
\begin{equation*} 
\varphi(x_1, \ldots, x_{\ell})=  \sum_{\substack{a \in \mathbb{N}^{\ell} \\ |a| \leq d }} c_ax^a,
\end{equation*} 
for some $d \in \mathbb{N}$.
  The norm map is then given by 
\begin{eqnarray*}
N: K(x_1, \ldots, x_{\ell}) & \to & \mathbb{Q}(x_1, \ldots, x_{\ell}) \\
 \sum_{\substack{a \in \mathbb{N}^{\ell} \\ \mid a \mid \leq d}} c_ax^a& \mapsto & \prod_{\sigma \in \Gal(L / \mathbb{Q})} \sum_{\substack{a \in \mathbb{N}^{\ell} \\ \mid a \mid \leq d}} \sigma(c_a)x^a.
\end{eqnarray*}
\end{definition}

Another important fact is the following: if $a \in \mathfrak{p}$, where $ \mathfrak{p}$ is a non-zero prime ideal in $\mathcal{O}_K$, then it holds that $N_{K/\mathbb{Q}}(a) \equiv 0 \mod p$ for $\mathfrak{p} \mid p$.

Let  $\zeta_K$ denote  the Dedekind zeta function over the number field $K$, so that $\zeta_{\mathbb{Q}} = \zeta$ is the classical Riemann zeta function. 
Let us denote by $\mathbb{Z}_p$ the ring of $p$-adic integers.

Let us define rectangular unimodular matrices.

\begin{definition}
Let  $\mathcal{R}$ be a domain and  $n <m \in \mathbb{N}$. Let $M \in \Mat_{n \times m}(\mathcal{R})$. $M$ is said to be rectangular unimodular, if there exist $m-n$ rows in $\mathcal R^m$, such that when adjoining these rows to $M$ the resulting $m \times m$ matrix $\overline{M}$ is invertible, i.e. $\det(\overline{M})$ is a unit in $\mathcal R$.
\end{definition}

From \cite{gustafson1981matrix} one has the following characterization of rectangular unimodular matrices over Dedekind domains;

\begin{proposition}
Let $\mathcal{D}$ be a Dedekind domain and $n < m \in \mathbb{N}$. Let $M \in \Mat_{n \times m}(\mathcal{D})$. $M$ is rectangular unimodular, if and only if the ideal generated by all the $n \times n$ minors of $M$ is $\mathcal{D}$.
\end{proposition}
For example coprime $m$-tuples over $\mathcal{O}_K$ are $1 \times m$ rectangular unimodular matrices, as $\mathcal{O}_K$ is a Dedekind domain. 
It is then immediate to characterize rectangular unimodular matrices in terms of prime ideals:
\begin{proposition}\label{characterization_unimodular}
Let $\mathcal{D}$ be a Dedekind domain and $n<m \in \mathbb{N}$.
$M \in \Mat_{n \times m}(\mathcal{D})$ is rectangular unimodular, if and only if $M \mod \mathfrak{p}$ has full rank for any $\mathfrak{p}$ non-zero prime ideal of $\mathcal{D}$.
\end{proposition}

\begin{definition}
The density of a set $S \subset \mathbb{Z}^d$ is defined to be 
\begin{equation*}
\rho(S) = \lim_{B \rightarrow \infty} \frac{|S \cap [-B,B[^d|}{(2B)^d}
\end{equation*}
if the limit exists. Then one defines the upper density $\bar{\rho}$ and the lower density $\underline{\rho}$ equivalently with the $\limsup$ and the $\liminf$ respectively. 
\end{definition}

We define the density over the ring of algebraic integers in the same way as in \cite{ferraguti2016mertens}.
\begin{definition}\label{definition_density}
Let  $\mathbb{E} = \{e_1, \ldots, e_k\}$ be an integral basis for $\mathcal{O}_K$.
Define
\begin{equation*}
\mathcal{O}[B, \mathbb{E}] = \left\{ \sum_{i=1}^k  a_ie_i \mid a_i \in [-B,B[ \cap \mathbb{Z} \right\}.
\end{equation*}
Let $m \in \mathbb{N}$ and  $S$ a subset of $\mathcal{O}_K^m$, we define the upper density of $S$ with respect to $\mathbb{E}$ to be 
\begin{equation*}
\bar{\rho}_{\mathbb{E}}(S) = \limsup_{B \to \infty} \frac{\mid \mathcal{O}[B,\mathbb{E}]^m \cap S \mid}{(2B)^{mk}}
\end{equation*}
and the lower density of $S$ with respect to $\mathbb{E}$ to be 
\begin{equation*}
\underline{\rho}_{\mathbb{E}}(S) = \liminf_{B \to \infty} \frac{\mid \mathcal{O}[B,\mathbb{E}]^m \cap S \mid}{(2B)^{mk}}.
\end{equation*}
We say that $S$ has density $d$ with respect to $\mathbb{E}$, if 
\begin{equation*}
\bar{\rho}_{\mathbb{E}}(S) = \underline{\rho}_{\mathbb{E}}(S) = \rho_{\mathbb{E}}(S) =  d.
\end{equation*}
\end{definition}
Whenever this density is independent of the chosen basis $\mathbb{E}$, we denote the density of a set $S$ by $\rho(S)$ and we say that $S$ has density $\rho(S)$.

For convenience, let us restate here the local to global principle by Poonen and Stoll in  \cite{bib:loctoglob}, since this will be the main tool in the proof of Theorem \ref{density_unimodular_matrices}. 
If $S$ is a set, then we denote by $2^S$ its powerset. Let $M_{\mathbb{Q}}= \{\infty \} \cup \{p \mid p \  \text{prime} \}$ be the set of all places of $\mathbb{Q}$, where we denote by $\infty$ the unique archimedean place of $\mathbb{Q}$. Let $\mu_{\infty}$ denote the Lebesgue measure on $\mathbb{R}^d$ and $\mu_p$ the normalized Haar measure on $\mathbb{Z}_p^d$.  For $T$ a subset of a metric space, let us denote by $\partial T$ its boundary.
By $\mathbb{R}_{\geq0}$ we denote the non-negative reals.
\begin{theorem}[\text{\cite[Lemma 1]{bib:loctoglob}}]\label{poonen}
Let $d$ be a positive integer. Let $U_{\infty} \subset \mathbb{R}^d$, such that $\mathbb{R}_{\geq0} \cdot U_{\infty} = U_{\infty}$ and $\mu_{\infty}(\partial(U_{\infty}))=0.$ Let $s_{\infty}= \frac{1}{2^d}\mu_{\infty}(U_{\infty} \cap [-1,1]^d)$. 
For each prime $p$, let $U_p \subset \mathbb{Z}_p^d$, such that $\mu_p(\partial(U_p)) =0$ and define $s_p = \mu_p(U_p)$. 
Define the following map 
\begin{eqnarray*}
P: \mathbb{Z}^d   &\rightarrow &  2^{M_{\mathbb{Q}}} \\
a  &\mapsto & \left\{ \nu \in M_{\mathbb{Q}} \mid a \in U_{\nu} \right\}.
\end{eqnarray*}
If the following is satisfied:
\begin{equation} \label{densitiycond}
\lim_{M \rightarrow \infty} \bar{\rho}\left( \left\{ a \in \mathbb{Z}^d \mid a \in U_p \  \text{for some prime} \ p > M \right\} \right)=0,
\end{equation}
then:
\begin{itemize}
\item[i)] $\sum_{\nu \in M_{\mathbb{Q}}} s_{\nu}$ converges.
\item[ii)] For $\mathcal{S} \subset 2^{M_{\mathbb{Q}}},$  $\rho(P^{-1}(\mathcal{S}))$ exists, and defines a measure on $2^{M_{\mathbb{Q}}}$.
\item[iii)] For each finite set $\{S\} \subset 2^{M_{\mathbb{Q}}}$, we have that
\begin{equation*}
\rho(P^{-1}(\{S\})) = \prod_{\nu \in S} s_{\nu} \prod_{\nu \not\in S} (1-s_{\nu}), 
\end{equation*}
and if $\mathcal{S}$ consists of infinite subsets of $2^{M_{\mathbb{Q}}}$, then $\rho(P^{-1}(\mathcal{S}))=0.$
\end{itemize}
\end{theorem}

To show  that \eqref{densitiycond} is satisfied, one can often apply the following useful lemma, that can be deduced from the result in \cite{torsten1991infinite}.

\begin{lemma}[\text{\cite[Lemma 2]{bib:loctoglob}}]\label{showdens}
Let $d$ and $M$ be positive integers. Let $f,g \in \mathbb{Z}[x_1, \ldots, x_d]$ be relatively prime. Define
\begin{equation*}
S_M(f,g) = \left\{ a \in \mathbb{Z}^d \mid  f(a) \equiv g(a) \equiv 0 \mod p \ \text{for some prime} \ p > M \right\},
\end{equation*}
then
\begin{equation*}
\lim_{M \rightarrow \infty} \bar{\rho}(S_M(f,g)) = 0. 
\end{equation*}
\end{lemma}

\section{The Density of Rectangular Unimodular Matrices}\label{sec3}

\begin{remark}\label{determinant_irreducible}
Let $\mathcal{R}$ be an integral domain and $n \in \mathbb{N}$.
Recall that, if  $X$ is an $n \times n$ polynomial matrix over $\mathcal{R}[x_{1,1}, \ldots, x_{n,n}]$ having as $(i,j)$ entry the variable $x_{i,j}$,
then $\det(X)  \in \mathcal{R}[x_{1,1}, \ldots, x_{n,n}]$ is irreducible.
\end{remark}

Let $\ell, k \in \mathbb{N}, f \in \mathbb{C}[x_1, \ldots, x_{\ell}]$ and $e=(e_1, \ldots, e_k) \in  \left( \mathbb{C} \setminus \{0\}\right)^k$.
In the new polynomial ring $\mathbb{C}[x_1^{(1)}, x_1^{(2)}, \ldots, x_{\ell}^{(k)}]$ let us denote by $f_e$
\begin{equation*}
f\left(\sum_{u=1}^k x_1^{(u)}e_u, \ldots, \sum_{u=1}^k x_{\ell}^{(u)}e_u\right).
\end{equation*}

\begin{lemma}\label{Basis_irreducible}
Let $\ell,k, f$ and $e$ be as above. If $f \in \mathbb{C}[x_1, \ldots, x_{\ell}]$ is irreducible, then $f_e$ is irreducible in $\mathbb{C}[x^{(1)}_1, x_1^{(2)}, \ldots, x^{(k)}_{\ell}]$. 
\end{lemma}

\begin{proof}
The idea of the proof is the following: we will assume the contrary, i.e. $f_e = g \cdot h$ and we will show that by the irreducibility of $f$ one can partition the variables appearing in $g$ and $h$. This allows us to evaluate in the variables annihilating $h$ (and therefore also $f_e$). On the other hand with the rest of the variables appearing in $f_e$, which are not fixed yet, we can yield $f_e$ to be non-zero, which brings the desired contradiction. In what follows we will do this in detail.

Let us fix the notation $x^{(w)} = \left( x_1^{(w)}, x_2^{(w)}, \ldots, x_{\ell}^{(w)} \right)$, for all $w \in \{1,\ldots, k\}$.
We suppose that
\begin{equation*}
f_e=g\left( x^{(1)}, \ldots, x^{(k)}\right)h\left( x^{(1)}, \ldots, x^{(k)}\right).
\end{equation*}
This factorization also holds in 
\begin{equation*}
S^{(1)} := \mathbb{C}\left(x^{(2)}, \ldots,  x^{(k)}\right)\left[x^{(1)} \right].
\end{equation*}
First we show that $f_e$ is irreducible in $S^{(1)}$: in  $S^{(1)}$ $f_e$ can be written as 
\begin{equation*}
f\left(e_1x_1^{(1)} + b_1, \ldots, e_1x_{\ell}^{(1)}+b_{\ell}\right),
\end{equation*}
where $b_i \in \mathbb{C}\left(x^{(2)}, \ldots,  x^{(k)}\right)$ for all $i \in \{1, \ldots, \ell\}$.
Since $e_1x_i^{(1)}+b_i$ is an affine transformation and $f$ is irreducible, we have that $f_e$ is irreducible  in $S^{(1)}$.
This forces, that for all $i \in \{1, \ldots, \ell\}$ the variable $x_i^{(1)}$  will not appear in $g$ and $h$ at the same time.
Repeating an analogous  argument for $S^{(w)} = \mathbb{C}\left(x^{(1)}, \ldots, x^{(w-1)}, x^{(w+1)}, \ldots, x^{(k)} \right) \left[ x^{(w)} \right]$ for all $w \in \{1, \ldots, k \}$ 
we can partition the blocks of variables $x^{(w)}$ for $g$ and $h$. So we can write
\begin{equation*}
f_e=g\left( x^{(u_1)},  \ldots, x^{(u_s)}\right)h\left( x^{(v_1)}, \ldots, x^{(v_t)}\right),
\end{equation*}
for some $s,t \in \mathbb{N}$. Without loss of generality we may assume that
\begin{equation*}
f_e=g\left( x^{(1)},  \ldots, x^{(s)}\right)h\left( x^{(s+1)}, \ldots, x^{(k)}\right).
\end{equation*}
By contradiction, let us assume that $f_e$ is reducible, so that we can choose a non-trivial factorization with $1\leq s< k$.
Now we evaluate $x^{(j)}$ at  $c^{(j)} = \left( c_1^{(j)}, \ldots, c_{\ell}^{(j)}  \right) \in \mathbb{C}^{\ell}$ for $j \in \{s+1,  \ldots, k \}$ in such a way, that 
\begin{equation*}
h\left(c^{(s+1)}, \ldots, c^{(k)}\right) = 0.
\end{equation*}
This is possible, since $\mathbb{C}$ is algebraically closed and $h$ is not a constant.
Thus, also $f_e$ evaluated at $c^{(s+1)},\dots, c^{(k)}$ is zero.

Observe that the maps
\begin{eqnarray*}
\psi_i: \mathbb{C}^k & \to & \mathbb{C} \\
\left(x_i^{(1)}, \ldots, x_i^{(k)}\right) & \mapsto & \sum_{j=1}^k e_jx_i^{(j)}
\end{eqnarray*}
are surjective for all $i \in \{1, \ldots, \ell\}$.
Then
\begin{equation*}
f_e\left(x_1^{(1)}, \ldots, x_{\ell}^{(k)}\right) =f\left(\psi_1\left(x_1^{(1)}, \ldots, x_1^{(k)}\right), \ldots, \psi_{\ell}\left(x_{\ell}^{(1)}, \ldots, x_{\ell}^{(k)}\right)\right).
\end{equation*}
Hence
\begin{eqnarray*}
0 &=& g\left( x^{(1)},  \ldots, x^{(s)}\right)h\left( c^{(s+1)}, \ldots, c^{(k)}\right) = f_e\left(x^{(1)},  \ldots, x^{(s)}, c^{(s+1)}, \ldots, c^{(k)}\right) \\
&=& f\left(\psi_1\left(x_1^{(1)}, \ldots, x_1^{(s)}, c_1^{(s+1)}, \ldots, c_1^{(k)}\right), \ldots, \psi_{\ell}\left(x_{\ell}^{(1)}, \ldots, x_{\ell}^{(s)}, c_{\ell}^{(s+1)}, \ldots, c_{\ell}^{(k)}\right)\right) \\
&=& f\left( \phi_1\left(x_1^{(1)}, \ldots, x_1^{(s)}\right), \ldots, \phi_{\ell}\left(x_{\ell}^{(1)}, \ldots, x_{\ell}^{(s)}\right)\right),
\end{eqnarray*}
where for all $ i \in \{1, \ldots, \ell\}$
\begin{equation*}
\phi_i\left(x_i^{(1)}, \ldots, x_i^{(s)}\right) := \psi_i\left(x_i^{(1)}, \ldots, x_i^{(s)}, c_i^{(s+1)}, \ldots, c_i^{(k)}\right),
\end{equation*}
which are again surjective, because $e_i \neq 0$, for all $i \in \{1, \ldots, k\}$. 
Let $z=(z_1, \ldots z_{\ell}) \in \mathbb{C}^{\ell}$ be such that $f(z) \neq 0.$
To get a contradiction, we want to find $ \left( a_j^{(1)}, \ldots, a_j^{(s)}  \right) \in \mathbb{C}^s$ for all $j \in \{1, \ldots, \ell \}$, such that
\begin{equation*}
f\left(\phi_1\left( a_1^{(1)}, \ldots, a_1^{(s)}\right), \ldots, \phi_{\ell}\left( a_{\ell}^{(1)}, \ldots, a_{\ell}^{(s)}\right)\right) \neq 0.
\end{equation*}
For this we use the surjectivity of $\phi_j$ for all $ j \in \{1, \ldots, \ell\}$, and that $f(z) \neq 0$:  we choose $ \left( a_j^{(1)}, \ldots, a_j^{(s)}  \right) \in \mathbb{C}^s$, such that
\begin{equation*}
\phi_j\left(a_j^{(1)}, \ldots, a_j^{(s)}\right) = z_j.
\end{equation*}
With this choice we got the desired contradiction, as
\begin{eqnarray*}
0 & \neq &  f\left(\phi_1\left( a_1^{(1)}, \ldots, a_1^{(s)}\right), \ldots, \phi_{\ell}\left( a_{\ell}^{(1)}, \ldots, a_{\ell}^{(s)}\right)\right) \\
&= & f_e\left(a^{(1)}, \ldots, a^{(s)}, c^{(s+1)}, \ldots, c^{(k)}\right) = 0.
\end{eqnarray*}
\end{proof}

\begin{remark}\label{sigma_irreducible}
Let $L$ be the Galois closure of $K/\mathbb Q$. Notice that, if $f(x_1, \ldots, x_{\ell})$ is irreducible in $\mathcal{O}_K[x_1, \ldots x_{\ell}]$, then we have that $\sigma(f)$ is irreducible in $ \sigma(\mathcal{O}_K)[x_1, \ldots x_{\ell}]$ for all $\sigma \in \Gal(L/  \mathbb{Q})$.
\end{remark}

\begin{lemma}\label{Norm_coprime}
Let $M$ be a positive integer and let $N$ be the norm map for the extension field  $K(x_1,\dots x_M)/\mathbb Q(x_1,\dots x_M)$. Let $F,G \in \mathcal{O}_K[x_1, \ldots, x_M]$ be irreducible and such that there are variables appearing in $F$ but not in $G$, then $N(F)$ and $N(G)$ are coprime.  
\end{lemma}

\begin{proof}
Let $L$ be the Galois closure of $K/ \mathbb{Q}$.
Let us assume that there exists a $h \mid N(F)$ and $h \mid N(G)$. By the definition of the norm over $K(x_1,\dots x_M)/\mathbb Q(x_1,\dots x_M)$, it follows that $h \mid \sigma(F)$ and $h \mid \tau(G)$, for some $\sigma, \tau \in \Gal(L/\mathbb{Q}). $ By Remark \ref{sigma_irreducible} we have that $\sigma(F)$ and $\tau(G)$ are irreducible in $ \sigma(\mathcal{O}_K)[x_1, \ldots x_M]$. Hence either $h=1$ (in which case we are done), or $h=\sigma(F)= \tau(G)$, which contradicts the assumption that there are variables of $F$ not appearing in $G$. 
\end{proof}

\begin{theorem}\label{density_unimodular_matrices}
Let $n$ and $m$ be positive integers such that $n<m$ and $K$ be an algebraic number field. The density of $n \times m $  rectangular unimodular matrices over $\mathcal{O}_K$ is
\[\prod_{i=0}^{n-1} \frac{1}{\zeta_K(m-i)},\]
where $\zeta_K$ denotes the Dedekind zeta function of $K$.
\end{theorem}

\begin{remark}
Notice that this is independent of the chosen basis. 
\end{remark}

\begin{proof}
To start the proof we need to set up a family of subsets $U_p$ of $\mathbb Z_p^{knm}$ for which condition \eqref{densitiycond} is verified, and such that $P^{-1}(\{\emptyset\})$ is the set of rectangular unimodular matrices over $\mathcal O_K$. In what follows, we actually construct their complements, that are nicer to deal with when it comes to Haar measures.

Let $k = [K: \mathbb{Q}]$ and $\{e_1, \ldots, e_k\}$ be an integral basis of $\mathcal{O}_K$. Let us consider the following maps: the usual reduction modulo a rational prime $p$
\begin{equation*}
\pi_p: \mathbb{Z}_p   \to   \mathbb{F}_p,
\end{equation*}
the bijective map 
\begin{eqnarray*}
\mathbb{E}_p: \mathbb{F}_p^k & \to & \mathcal{O}_K/(p) \\
(a_1, \ldots, a_k) & \mapsto & \sum_{i=1}^k a_ie_i,
\end{eqnarray*}
and the natural surjective map
\begin{equation*}
\psi_p: \mathcal{O}_K/(p)   \to   \prod_{p \mid \mathfrak{p}} \left( \mathcal{O}_K/ \mathfrak{p}\right).
\end{equation*}
Observe that these maps can be extended componentwise to the following maps:
\begin{eqnarray*}
\overline{\pi}_p: \mathbb{Z}_p^{knm} & \to &  \mathbb{F}_p^{knm},
\\
\overline{\mathbb{E}}_p: \mathbb{F}_p^{knm} &\to &  (\mathcal{O}_K/(p))^{n \times m}, \\
\overline{\psi}_p: (\mathcal{O}_K/(p))^{n \times m} &\to &  \prod_{ \mathfrak{p} \mid p} \left(\mathcal{O}_K/ \mathfrak{p}\right)^{n \times m}.
\end{eqnarray*}
We get the following composition $F_p= \overline{\psi}_p \circ \overline{\mathbb{E}}_p \circ \overline{\pi}_p$:
\begin{equation*}
 \mathbb{Z}_p^{knm} \overset{\overline{\pi}_p}{\longrightarrow} \mathbb{F}_p^{knm} \overset{\overline{\mathbb{E}}_p}{\longrightarrow } (\mathcal{O}_K/(p))^{n \times m} \overset{\overline{\psi}_p}{\longrightarrow} \prod_{\mathfrak{p} \mid p}( \mathcal{O}_K/\mathfrak{p})^{n \times m}.
\end{equation*}
Observe that $F_p$ is a composition of surjective maps, hence also surjective.
Let $\{\mathfrak{p}_1, \ldots, \mathfrak{p}_{\ell_p}\}$ be the set of prime ideals of $\mathcal O_K$ lying above $p$. Let $T_p = \prod_{\mathfrak{p} \mid p}( \mathcal{O}_K/\mathfrak{p})^{n \times m}$.
To easier the notation, let us define the following subset of $T_p$
\begin{equation*}
\mathcal{L}_p = \left\{ \left(a_{\mathfrak{p}_1}, \ldots, a_{\mathfrak{p}_{\ell_p}} \right) \in T_p \mid a_{\mathfrak{p}_i} \in \mathbb{F}_{p^{\deg(\mathfrak{p}_i)}}^{n \times m} \ \text{has full rank} \right\}.
\end{equation*}

Consider now the following set
\begin{equation*}
A_p = \left\{ A \in \mathbb{Z}_p^{knm}  \mid F_p(A) \in \mathcal{L}_p   \right\}.
\end{equation*}

We are interested in computing the Haar measure of $A_p$, which will be needed in the last part of the proof.
Notice that 
\begin{equation*}
A_p = \bigsqcup_{f \in \mathcal{L}_p} F_p^{-1}(f).
\end{equation*}
Hence we have that
\begin{equation*}
\mu_p(A_p) = \sum_{f \in \mathcal{L}_p} \mu_p(F_p^{-1}(f)).
\end{equation*}

We first want to compute $\mu_p(F_p^{-1}(f))$: let $f$ be in $T_p$, then observe that we can write 
\begin{equation*}
\mu_p(F_p^{-1}(f)) = \mu_p(\overline{\pi}_p^{-1} \overline{\mathbb{E}}_p^{-1}\overline{\psi}_p^{-1}(f)). 
\end{equation*}
Since $\overline{\psi}_p$ is $\mathbb{F}_p$-linear, there exists an $\bar{f} \in (\mathcal{O}_K/(p))^{n \times m}$ with
\begin{equation*}
\overline{\psi}_p^{-1}(f)) = \bar{f} + \Ker(\overline{\psi}_p).
\end{equation*}
Hence we can write
\begin{equation*}
\mu_p(F_p^{-1}(f) = \mu_p(\overline{\pi}_p^{-1}( \overline{\mathbb{E}}_p^{-1}(\bar{f} + \Ker(\overline{\psi}_p)))). 
\end{equation*}
Since $\overline{\mathbb{E}}_p$ is bijective and the kernel of $\overline{\pi}_p$ is $p\mathbb{Z}_p^{knm}$, there exist $f_1, \ldots, f_{\mid \Ker(\overline{\psi}_p) \mid } \in \mathbb{Z}_p^{knm}$, such that we can write
\begin{equation*}
F_p^{-1}(f) = \bigsqcup_{i=1}^{ \mid \Ker(\overline{\psi}_p) \mid} \left( f_i + p \mathbb{Z}_p^{knm} \right).
\end{equation*}
Notice that the Haar measure of  $\left( f_i + p \mathbb{Z}_p^{knm} \right)$ is $p^{-knm}$, independently of $f_i$. 
Hence the Haar measure of $F_p^{-1}(f)$ is
\begin{equation*}
\mu_p( F_p^{-1}(f) ) = \mid \Ker( \overline{\psi}_p) \mid p^{-knm}.
\end{equation*}
Since this is independent of the chosen $f \in T_p$, one has that
\begin{equation*}
\mu_p( A_p) = \sum_{f \in \mathcal{L}_p} \mu_p(F_p^{-1}(f))= \mid \mathcal{L}_p \mid \cdot   \mid \Ker( \overline{\psi}_p) \mid p^{-knm}.
\end{equation*}
Let us first compute the size of the the kernel of  $\overline{\psi}_p$.
Since  
\begin{equation*}
\psi_p: \mathcal{O}_K/(p) \to \prod_{\mathfrak{p} \mid p} \mathcal{O}_K/ \mathfrak{p}
\end{equation*}
 is surjective, we have that
\begin{equation*}
\dim_{\mathbb{F}_p}\left(\Ker(\psi_p)\right) = \dim_{\mathbb{F}_p}\left( \mathcal{O}_K/(p)\right) - \dim_{\mathbb{F}_p}\left( \prod_{\mathfrak{p} \mid p} \mathcal{O}_K/ \mathfrak{p}\right) = k - \sum_{\mathfrak{p} \mid p} \deg(\mathfrak{p}).
\end{equation*}
Hence for  
\begin{equation*}
\overline{\psi}_p: (\mathcal{O}_K/(p))^{n \times m} \to \prod_{\mathfrak{p} \mid p} (\mathcal{O}_K/ \mathfrak{p})^{n \times m}
\end{equation*}
we have that 
\begin{equation*}
\dim_{\mathbb{F}_p}\left(\Ker(\overline{\psi}_p)\right) = knm - \sum_{\mathfrak{p} \mid p} \deg(\mathfrak{p}) nm.
\end{equation*}
Therefore 
\begin{equation*}
\mid \Ker(\overline{\psi}_p) \mid = p^{knm - \sum_{\mathfrak{p} \mid p} \deg(\mathfrak{p})nm}.
\end{equation*}

The next step is to compute the size of $\mathcal{L}_p$. 
Since $\mathcal{O}_K/ \mathfrak{p} \cong \mathbb{F}_{p^{\deg(\mathfrak{p})}}$ and for a prime power $q$ the number of full rank matrices over $\mathbb{F}_q^{n \times m}$ is 
\begin{equation*}
\prod_{i=0}^{n-1} \left(q^m-q^i\right),
\end{equation*}
we have that 
\begin{equation*}
\mid \mathcal{L}_p \mid = \prod_{\mathfrak{p} \mid p} \prod_{i=0}^{n-1} \left(  p^{\deg(\mathfrak{p})m}- p^{\deg(\mathfrak{p})i} \right).
\end{equation*}

Thus the Haar measure of $A_p$ is given by 
\begin{eqnarray*}
 \mu_p \left( A_p \right) &= &
 p^{-knm} \mid \Ker( \overline{\psi}_p) \mid \cdot \mid \mathcal{L}_p \mid \\  
   &=&  \frac{1}{p^{knm}}  \left( p^{knm - \sum_{ \mathfrak{p} \mid p}
 \deg(\mathfrak{p})nm} \right)  
    \prod_{\mathfrak{p} \mid p} 
   \prod_{i=0}^{n-1} \left(  p^{\deg(\mathfrak{p})m}-
  p^{\deg(\mathfrak{p})i} \right)\\  
&=&p^{ - \sum_{ \mathfrak{p} \mid p}
 \deg(\mathfrak{p})nm}   \prod_{\mathfrak{p} \mid p} \prod_{i=0}^{n-1} \left(  p^{\deg(\mathfrak{p})m}- p^{\deg(\mathfrak{p})i} \right)  \\
%&=&\prod_{ \mathfrak{p} \mid p} p^{ - \deg(\mathfrak{p})nm}   \prod_{\mathfrak{p} \mid p} \prod_{i=0}^{n-1} \left(  p^{\deg(\mathfrak{p})m}- p^{\deg(\mathfrak{p})i} \right)  \\ 
&=&  \prod_{\mathfrak{p} \mid p} \prod_{i=0}^{n-1}  p^{-\deg(\mathfrak{p})m} \left( p^{\deg(\mathfrak{p})m}- p^{\deg(\mathfrak{p})i} \right)   \\
&=&  \prod_{\mathfrak{p} \mid p} \prod_{i=0}^{n-1}  \left( 1- p^{\deg(\mathfrak{p})(i-m)} \right) .
\end{eqnarray*}

Finally we are ready to use Theorem \ref{poonen} to compute the wanted density. Let now $U$ be the set of rectangular unimodular $n \times m$ matrices over $\mathcal{O}_K$, notice that thanks to Proposition \ref{characterization_unimodular} we can write 
\begin{equation*}
U = \left\{M \in \Mat_{n \times m}(\mathcal{O}_K) \mid M \bmod \mathfrak{p}   \ \text{ has full rank for any prime ideal} \ \mathfrak{p} \subset \mathcal{O}_K \right\}.
\end{equation*}

Since we want to use Theorem \ref{poonen}, we want to choose a family $\{U_\nu\}_{\nu \in M_{\mathbb{Q}}}$ such that for $\mathcal{S} = \{\emptyset\}$, we have that $P^{-1}(\{\emptyset\}) = U$.
For this, let us choose $U_{\infty} = \emptyset$, then clearly $s_{\infty} = 0$. 
Let $U_p = A_p^c = \mathbb{Z}_p^{knm} \setminus A_p$.
We have that 
\begin{equation*}
s_p = \mu_p(U_p) = 1- \mu_p(A_p)= 1-\prod_{\mathfrak{p} \mid p} \prod_{i=0}^{n-1}  \left( 1- p^{\deg(\mathfrak{p})(i-m)} \right).
\end{equation*}
from which it follows that, if \eqref{densitiycond} holds, using Theorem \ref{poonen}, we get
\begin{eqnarray*}
\rho(U)= \rho(P^{-1}(\{\emptyset\}) &=& \prod_{\nu \in \emptyset} s_{\nu} \prod_{\nu \not\in \emptyset} (1-s_{\nu}) \\
&=& (1- s_{\infty}) \prod_{p \ \text{prime}} (1-s_p) \\
&=&  \prod_{p \ \text{prime}} \prod_{\mathfrak{p} \mid p} \prod_{i=0}^{n-1}  \left( 1- p^{\deg(\mathfrak{p})(i-m)} \right) \\
& =&   \prod_{i=0}^{n-1} \prod_{p \ \text{prime}} \prod_{\mathfrak{p} \mid p}  \left( 1 - \frac{1}{p^{\deg(\mathfrak{p})(m-i)} } \right) \\
&=&  \prod_{i=0}^{n-1} \frac{1}{\zeta_K(m-i)},
\end{eqnarray*}
which is exactly the final claim. Therefore, it remains to show that \eqref{densitiycond} holds.

Let 
\begin{eqnarray*}
\mathbb{E}: \mathbb{Z}^{k} &\to& \mathcal{O}_K, \\
(a_1, \ldots, a_k) & \mapsto & \sum_{i=1}^k a_ie_i.
\end{eqnarray*}
and $\overline{\mathbb{E}}$ be its natural componentwise extension from $\mathbb Z^{knm}$ to $\mathcal{O}_K^{n \times m}$.

For $A \in \mathbb{Z}^{knm}$, let us denote the $ n \times n$ minors of $\overline{\mathbb{E}}(A)$ by $A_i$ for $i \in \left\{ 1, \ldots, {m\choose n}  \right\}$. Hence $A \in U_p$ is equivalent to $\langle A_1, \ldots, A_{{m\choose n}} \rangle \subseteq \mathfrak{p}$ for some $\mathfrak{p} \mid p$. Thus for all $i \in \left\{ 1, \ldots, {m\choose n}  \right\}$ we have that $\langle A_i \rangle \subseteq\mathfrak{p}$ and hence that $N_{K / \mathbb{Q}}(A_i) \equiv 0 \mod p$. This implies that  $p \mid N_{K / \mathbb{Q}}(A_i)$ for all $i \in \left\{ 1, \ldots, {m\choose n}  \right\}$  and hence that 
\begin{equation*}
p \mid \gcd\left(N_{K / \mathbb{Q}}(A_1), \ldots, N_{K / \mathbb{Q}}(A_{{m\choose n}})\right).
\end{equation*}

Observe that by the consideration above, the set
\begin{equation*}
 \left\{ A \in \mathbb{Z}^{knm} \mid     A \in U_p \ \text{for some prime} \ p>M \right\}
\end{equation*}
is a subset of 
\begin{equation*}
B_M=\left\{ A \in \mathbb{Z}^{knm} \mid    p \mid N_{K / \mathbb{Q}}(A_1), p \mid N_{K / \mathbb{Q}}(A_2)  \ \text{for some prime} \ p>M \right\} .
\end{equation*}
For \eqref{densitiycond} to be satisfied, it remains to show that $B_M$ has upper density tending to zero, for $M$ going to infinity. In order to prove that, we want to use Lemma \ref{showdens} but for that we need a pair of coprime polynomials.
Let us consider the matrix $B\in \mathbb Q[x_{1,1}, x_{1,2},\dots, x_{n,m}]$ having as $(i,j)$-th entry the variable  $x_{i,j}$.
%)_{\substack{1 \leq i\leq n \\ 1 \leq j \leq m}}$  
Using Remark \ref{determinant_irreducible} we know that for all $i \in \left\{ 1, \ldots, {m\choose n}  \right\}$,  the $n \times n$ minors $m_i$ of $A$ are  irreducible  in $\mathbb{C}[x_{1,1}, \ldots, x_{n,m}]$.
Recall that for an integral basis $(e_1, \ldots, e_k)$ for $\mathcal{O}_K$, with Lemma \ref{Basis_irreducible} we have that for all $i \in \left\{ 1, \ldots, {m\choose n}  \right\}$
\begin{equation*}
m_i \left( \sum_{u=1}^k x_{1,1}^{(u)} e_u, \ldots,\sum_{u=1}^k x_{n,m}^{(u)} e_u \right) = (m_i)_e
\end{equation*}
are also irreducible in $\mathcal{O}_K[x_{1,1}^{(1)}, x_{1,1}^{(2)}, \ldots, x_{n,m}^{(k)}]$. Since $(m_1)_e$ and $(m_2)_e$ have different variables appearing, we can use Lemma \ref{Norm_coprime} to conclude that $f=N((m_1)_e)$ and $g=N((m_2)_e)$ are coprime polynomials in $\mathbb{Z}[x_{1,1}^{(1)}, x_{1,1}^{(2)}, \ldots, x_{n,m}^{(k)}]$. 
Observe now that since  $ N_{K / \mathbb{Q}}(A_i) =  N((m_i)_e(A))$, for $i \in \{1,2\}$ we have that
\begin{equation*}
B_M=\left\{ A \in \mathbb{Z}^{knm} \mid    p \mid f(A), p \mid g(A)  \ \text{for some prime} \ p>M \right\} .
\end{equation*}
Hence, using Lemma \ref{showdens} directly we have proven \eqref{densitiycond}, i.e. $\lim_{M \to \infty} \bar{\rho} (B_M) =0 $.

\end{proof}

\begin{remark}\label{daniels_remark}
It is worth noticing that instead of using Lemma \ref{Basis_irreducible} and \ref{Norm_coprime}, one could use a combination of \cite[Proposition 3.2]{bright2016failures}, \cite[Lemma 3.3]{bright2016failures} (and therefore \cite[Theorem 18]{bhargava2015geometry}) in order to prove that condition \eqref{densitiycond} holds.
\end{remark}

As a corollary, setting $n=1$ one gets the generalisations of the Mertens-Ces\'{a}ro Theorem, which coincides with the result in \cite{ferraguti2016mertens} and does not use the lattice based argument \cite[Theorem 3.10]{ferraguti2016mertens}:

\begin{corollary}
For $m \in \mathbb{N}$, the density of coprime $m$-tuples over the algebraic integers $\mathcal{O}_K$ of a number field $K$ is $\frac{1}{\zeta_K(m)}$. 
\end{corollary}

And as another immediate corollary of this result, by setting $K= \mathbb{Q}$, one gets a fix for the result in \cite{bib:maze2011natural}:

\begin{corollary}
For $n<m \in \mathbb{N}$, the density of $n \times m$ rectangular unimodular matrices over $\mathbb{Z}$ is given by
\begin{equation*}
\prod_{i=0}^{n-1} \frac{1}{\zeta(m-i)},
\end{equation*}
where $\zeta$ is the Riemann zeta function.
\end{corollary}

\section{Experimental results}\label{sec4}

Since Theorem \ref{density_unimodular_matrices} gives only the asymptotic density of the set of rectangular unimodular matrices, we include here some experiments to see how the asymptotics agree with the actual estimates: in the Tables \ref{table_over_integers}, \ref{table_over_sqrt2}, \ref{table_over_sqrt-3},  \ref{table_over_x^5-13x-7},  we performed a Monte-Carlo test on some number fields to understand the accuracy of the result in Theorem \ref{density_unimodular_matrices}. 
In the following tables we fixed the bound on the coefficients of the basis, denoted by $B$ (like in Definition \ref{definition_density}) to be just $3$ and run a 
Monte-Carlo test on $50000$ random matrices over the algebraic integers of the respective number field to check how many are unimodular, the result of this test will be denoted by $\mathbb{P}(U)$. As one can see, it turns out that already for $B=3$ and when the degree of $K$ is greater than or equal to $2$ we have an agreement of around $99\%$.
It seems that for the case of rational integers we have one of the worst situations, with an accuracy of around $90\%$ in the worst case $n=1$ $m=2$, $95\%$ for $n=2 $ $m=3$, and $99\%$ in all other cases.
It is also worth noticing that the accuracy seems to be increasing when the number of columns grows with respect to the number of rows.
For the case of rational integers, in Figure \ref{Graph} we see how quickly the asymptotic density is achieved for fixed values for $n$ and $m$ when the bound $B$ on the coordinates increases. 

The SAGE code is available upon request.

\begin{table}[H]
\begin{subtable}{5 cm}
\begin{tabular}{c|c|c|c}
	$n$ & $m$ & $\mathbb{P}(U)$ & predicted density \\
	\hline
  1& 2&  0.66896 &          0.60792 \\
  1& 3&  0.84234 &          0.83190 \\
  2& 3&  0.53026 &          0.50573 \\
  1& 4&  0.92652 &          0.92393 \\
  2& 4&  0.77268 &          0.76863 \\
  3& 4&  0.47866 &          0.46727 \\
  1& 5&  0.96530 &          0.96438 \\
  2& 5&  0.89246 &          0.89103 \\
  3& 5&  0.74122 &          0.74125 \\
  4& 5&  0.45820 &          0.45063 \\
   \hline
\end{tabular}
\caption{Comparing $\mathbb{P}(U)$ with the predicted density from Theorem \ref{density_unimodular_matrices} over $\mathbb Z$}\label{table_over_integers}
\end{subtable}
\quad \quad \quad \quad \quad
\begin{subtable}{5cm}
\begin{tabular}{c|c|c|c}
	$n$ & $m$ & $\mathbb{P}(U)$ & predicted density \\
	\hline
  1&  2&   0.70308           & 0.69687 \\
  1&  3&   0.86826           & 0.86803 \\
  2&  3&   0.60542           & 0.60491 \\
  1&  4&   0.93628           & 0.93654 \\
  2&  4&   0.81370           & 0.81295 \\
  3&  4&   0.56878           & 0.56652 \\
  1&  5&   0.96812           & 0.96861 \\
  2&  5&   0.90542           & 0.90714 \\
  3&  5&   0.79024           & 0.78743 \\
  4&  5&   0.54838           & 0.54874 \\
\hline
\end{tabular}
\caption{Comparing $\mathbb{P}(U)$ with the predicted density from Theorem \ref{density_unimodular_matrices} over the ring of integers of the number field $\mathbb{Q}(\sqrt{2})$}\label{table_over_sqrt2}
\end{subtable}
\end{table}

\begin{table}[H]
\begin{subtable}{5 cm}
\begin{tabular}{c|c|c|c}
	$n$ & $m$ & $\mathbb{P}(U)$ & predicted density \\
	\hline
1 & 2 & 0.84256 & 0.83941 \\
1 & 3 & 0.95892 & 0.95813\\
2 & 3 & 0.80680 & 0.80427\\
1 & 4 & 0.98754 & 0.98713\\
2 & 4 & 0.94656 & 0.94581\\
3 & 4 & 0.79298 & 0.79393\\
1 & 5 & 0.99620 & 0.99582\\
2 & 5 & 0.98416 & 0.98302\\
3 & 5 & 0.94072 & 0.94186\\
4 & 5 & 0.79192 & 0.79061\\
\hline
\end{tabular}
\caption{Comparing $\mathbb{P}(U)$ with the predicted density from Theorem \ref{density_unimodular_matrices} over the ring of integers of the number field $ \mathbb Q[x]/(x^3+x+1)$}\label{table_over_sqrt-3}
\end{subtable}
\quad \quad \quad \quad \quad
\begin{subtable}{5cm}
\begin{tabular}{c|c|c|c}
	$n$ & $m$ & $\mathbb{P}(U)$ & predicted density \\
	\hline
  1 &  2 &   0.83346 &            0.84149\\ 
  1 &  3 &   0.96832 &            0.97000\\
  2 &  3 &   0.81210 &            0.81625\\
  1 &  4 &   0.99328 &            0.99371\\
  2 &  4 &   0.96516 &            0.96391\\
  3 &  4 &   0.81040 &            0.81112\\
  1 &  5 &   0.99876 &            0.99860\\
  2 &  5 &   0.99244 &            0.99232\\
  3 &  5 &   0.96278 &            0.96256\\
  4 &  5 &   0.80914 &            0.80999\\
   \hline
\end{tabular}
\caption{Comparing $\mathbb{P}(U)$ with the predicted density from Theorem \ref{density_unimodular_matrices}  over the ring of integers of the number field $\mathbb{Q}[x]/(x^5-13x-7)$}\label{table_over_x^5-13x-7}
\end{subtable}
\end{table}

\newpage

\begin{figure}[H]
\includegraphics[width=10cm]{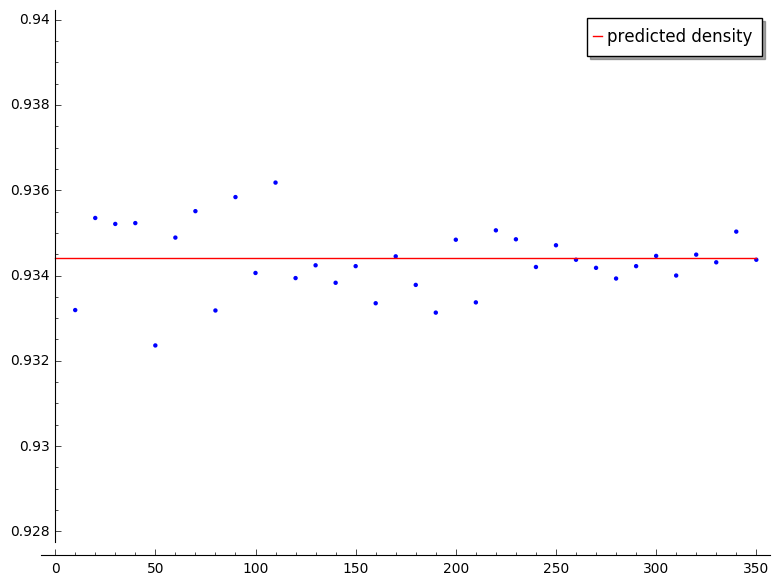}
\caption{Testing 100000 random matrices over the rational integers with fixed $n=5, m= 9$, letting $B$ vary from $10$ to $350$}
\label{Graph}
\end{figure}

\section*{Acknowledgments}
The first author is thankful to the Swiss National Science Foundation grant number 171248. The second author has been partially supported by the Swiss National Science Foundation under grant number 169510. 
We would also like to thank Daniel Loughran for pointing out Remark \ref{daniels_remark} and for interesting discussions and suggestions.

\bibliographystyle{plain}
\bibliography{biblio}

\end{document}